\documentclass[reqno]{amsart}

\input xy
\xyoption{all}
\usepackage{epsfig}
\usepackage{color}
\usepackage{amsthm}
\usepackage{amssymb}
\usepackage{amsmath}
\usepackage{amscd}
\usepackage{amsopn}
\usepackage{url}

\usepackage{hyperref}\hypersetup{colorlinks}

\usepackage{color} %\textcolor{red}{Test}

\definecolor{darkred}{rgb}{1,0,0} %can change the intensity in [0,1]
\definecolor{darkgreen}{rgb}{0,0.8,0}
\definecolor{darkblue}{rgb}{0,0,1}

\hypersetup{colorlinks,
linkcolor=darkblue,
filecolor=darkgreen,
urlcolor=darkred,
citecolor=darkgreen}

\numberwithin{equation}{section}
\newtheorem {Theorem}{Theorem}
\numberwithin{Theorem}{section}

\newtheorem {Lemma}[Theorem]    {Lemma}

\newtheorem {Proposition}[Theorem]{Proposition}

\theoremstyle{definition}

\theoremstyle{remark}

\expandafter\chardef\csname pre amssym.def at\endcsname=\the\catcode`\@
\catcode`\@=11
\def\undefine#1{\let#1\undefined}
\def\newsymbol#1#2#3#4#5{\let\next@\relax
 \ifnum#2=\@ne\let\next@\msafam@\else
 \ifnum#2=\tw@\let\next@\msbfam@\fi\fi
 \mathchardef#1="#3\next@#4#5}
\def\mathhexbox@#1#2#3{\relax
 \ifmmode\mathpalette{}{\m@th\mathchar"#1#2#3}%
 \else\leavevmode\hbox{$\m@th\mathchar"#1#2#3$}\fi}
\def\hexnumber@#1{\ifcase#1 0\or 1\or 2\or 3\or 4\or 5\or 6\or 7\or 8\or
 9\or A\or B\or C\or D\or E\or F\fi}

\font\teneufm=eufm10
\font\seveneufm=eufm7
\font\fiveeufm=eufm5
\newfam\eufmfam
\textfont\eufmfam=\teneufm
\scriptfont\eufmfam=\seveneufm
\scriptscriptfont\eufmfam=\fiveeufm

\catcode`\@=\csname pre amssym.def at\endcsname

\newcommand{\CR}{{\mathcal R}}

\newcommand{\Aa}{{\mathcal A}}

\newcommand{\Ii}{{\mathcal I}}

\newcommand{\Pp}{{\mathcal P}}

\newcommand{\Ss}{{\mathcal S}}

\def    \R      {{\mathbb R}}

\def    \Z      {{\mathbb Z}}
\def    \N      {{\mathbb N}}

\def    \T      {{\mathbb T}}
\def    \CP     {{\mathbb C}{\mathbb P}}

\def    \12    {{\frac{1}{2}}}

\def    \loc  {\mathrm{loc}}

\def    \HF     {\operatorname{HF}}
\def    \FH     {\operatorname{FH}}

\def    \CF     {\operatorname{CF}}

\def    \Fix     {\operatorname{Fix}}
\def    \Per     {\operatorname{Per}}

\def    \va     {\vec{a}}

\begin{document}

%%%%%%%%%%%%%%%%%%%%%%%%%%%%%%
%   TEXT FORMATTING

\setlength{\smallskipamount}{6pt}
\setlength{\medskipamount}{10pt}
\setlength{\bigskipamount}{16pt}

%%%%%%%%%%%%%%%%%%%%%%%%%%

%%%%%%%%%%%%%%%%%%%%%%%%%%

%%%%%%%%%%%           BEGINNING OF  TEXT

%%%%%%%%%%%%%%%%%%%%%%%%%%

\title[Primes and period growth]{On primes and period growth for Hamiltonian diffeomorphisms}

\author[Ely Kerman]{Ely Kerman}

\address{ Department of Mathematics, University of Illinois at 
Urbana-Champaign, Urbana, IL 61801, USA}
\email{ekerman@math.uiuc.edu}

%\keywords{Periodic points, Hamiltonian diffeomorphisms, Floer homology}
\date{\today} \thanks{This work was partially supported by a grant
from the Simons Foundation (207839). The author also acknowledges support from National Science
Foundation grant DMS 08-38434 ÓEMSW21-MCTP: Research Experience for
Graduate Students.}

\begin{abstract}  
Here we use Vinogradov's prime distribution theorem and a multi-dimensional generalization 
due to Harman to  strengthen some recent results from \cite{gg2} and \cite{gk} concerning the periodic points of Hamiltonian diffeomorphisms. In particular we establish resonance relations for 
the mean indices of the fixed points of Hamiltonian diffeomorphisms which do not have periodic points with arbitrarily large 
periods  in $\mathbb{P}^2$,  the set of natural numbers greater than one which have at most two prime factors when counted with multiplicity. As an application of these results we extend the methods of \cite{res1} to partially recover, using only symplectic tools,  a theorem on the periodic points of Hamiltonian diffeomorphisms of the sphere by  Franks and Handel from \cite{fh}. 

\end{abstract}

\maketitle

\section{Introduction and main results }
\subsection{Introduction}
\label{sec:intro}

Hamiltonian diffeomorphisms tend to have infinitely many periodic points as well as periodic points with arbitrarily large periods.  For Hamiltonian diffeomorphisms of the two-dimensional sphere these \emph{tendencies} are illustrated by the following two beautiful theorems by Franks, and Franks and Handel. 
\begin{Theorem}( \cite{f1,f2})
\label{f}
If $\phi \colon S^2 \to S^2$ is a  Hamiltonian diffeomorphism then $\phi$ has either two or infinitely many periodic points. 
\end{Theorem}
\begin{Theorem}(Theorem 1.4 of \cite{fh})
\label{fh}
If $\phi \colon S^2 \to S^2$ is a nontrivial Hamiltonian diffeomorphism with at least three 
fixed points, then there exists integers $n>0$ and $K>0$ such that $\phi^n$ has a periodic point of period $k$ for every $k \geq K$.
\end{Theorem}
.

It is expected that results similar to Theorems \ref{f} and \ref{fh}  hold for the Hamiltonian diffeomorphisms of symplectic manifolds, such as $\CP^n$, with positive first Chern number. With such generalizations  in mind,  a new  proof of Theorem \ref{f} which uses only tools from symplectic topology was established in \cite{res1}. The initial goal of this paper was to extend these methods to obtain 
a symplectic proof of Theorem \ref{fh}. Here we fall short of this goal. Most significantly, it is not clear to the author how to deal with nonisolated fixed points using the methods of \cite{res1}, and so we need to assume more than the simple hypothesis of nontriviality  from Theorem \cite{fh}. Even with a stronger hypothesis we do not recover the  linear period growth of \ref{fh}. Instead we prove the following result.

 \begin{Theorem}\label{fhnew}
Let $\phi$ be a Hamiltonian diffeomorphism of $S^2$ with at least three fixed points such that 
 $\Fix(\phi^{k})$ is finite for all $k \in \N$. Then $\phi$ has fixed points of arbitrarily large period.
More precisely, either $\phi$ has infinitely many periodic points with arbitrarily large period and rational mean indices not equal to zero modulo four, or there exists a $K \in \N$ such that $\phi^K$ has infinitely many periodic points with arbitrarily large periods in $\mathbb{P}^2$, the set of natural numbers greater than one which have at most two prime factors when counted with multiplicity.
\end{Theorem}

Two new results of some independent interest and more general scope arise from the proof of Theorem \ref{fhnew}. 
 Both of these  are established using deep theorems from number theory concerning the 
distribution of prime numbers.  More precisely, using Vinogradov's prime distribution theorem (Theorem \ref{prime} below), we prove the 
following sharper version of Theorem 1.18 (part (i)) from \cite{gg2} where  the existence of infinitely many contractible periodic points is established under the same hypotheses.
\begin{Theorem}
\label{118plus}
Suppose that $(M, \omega)$ is a closed, weakly-monotone and rational symplectic manifold. Let $\varphi$ be a Hamiltonian diffeomorphism of $(M, \omega)$ which has finitely many contractible fixed points one of which is a symplectically degenerate maximum. Then $\varphi$ has infinitely many distinct  contractible periodic points of arbitrarily large prime period. 
\end{Theorem}

Using Theorem  \ref{118plus} together with a multidimension generalization of Vinogradov's theorem due to Harman (Theorem \ref{kprime} below), we then augment a result  from \cite{gk} to establish the existence of resonance relations for the fixed points of certain Hamiltonian diffeomorphisms. 

\begin{Theorem}
\label{resplus}
Suppose that $(M, \omega)$ is a closed, weakly-monotone and rational symplectic manifolds whose minimal Chern number $N$ is greater than half its dimension. 
Let $\varphi$ be a Hamiltonian diffeomorphism  of $(M, \omega)$ which has finitely many contractible fixed points
and does not have contractible periodic points of arbitrarily large period in $\mathbb{P}^2$.
Then the nonzero mean indices  of the contractible fixed points  of $\varphi$ satisfy at least one nontrivial resonance relation
(i.e., they satisfy a nontrivial linear equation with integer coefficients which is homogeneous modulo $2N$).
\end{Theorem}

The relevant theorem from \cite{gk} establishes identical restrictions on the periodic point set  of Hamiltonian diffeomorphisms with finitely many such points. A fundamental example covered by Theorem  \ref{resplus} and not by \cite{gk} is that of a nontrivial rotation of $S^2$ by a rational multiple of $2\pi$. 

The effectiveness of these number theoretic results on the distributions of primes in extending the theorems from \cite{gg2} and \cite{gk} is the consequence of two 
factors. The first  of these is the simple fact that when considering the growth of periods for the  periodic points of a map $\phi$ the following alternative is often useful; a prime iterate of $\phi$, say $\phi^{p}$, either has the same fixed point set as $\phi$, or $\phi$ has a periodic point of period $p$. The other factor has to do with a common feature  of the proofs of the theorems from \cite{gg2,gk}. For a Hamiltonian diffeomorphism of a rational symplectic manifold $(M, \omega)$ whose minimal Chern number $N$ is finite,  the actions of contractible fixed points are defined modulo the index of rationality,  and their mean indices are defined modulo $2N$.  When there are finitely many such fixed points, the collection of their actions or mean indices can then be viewed as points on a torus.  Since the action and mean index of a fixed point both grow linearly under iteration,  under suitable assumptions, the collections of actions and mean indices  follow linear paths on a fixed torus as the map is iterated. This scenario occurs in both \cite{gg2} and \cite{gk}  where restrictions on these linear paths, coming from the insights at the hearts of the proofs of the Conley conjecture  from \cite{Hi, Gi:conley}, are then used to infer the desired conclusions. The results of Vinogradov and Harman mentioned above allow us to detect  the same restrictions on these linear paths while at the same time exploiting the advantage of considering only prime iterations.

\subsection{Organization} In the next section we recall the theorems from number theory and the symplectic tools needed in the rest of the paper. The proofs of Theorem \ref{118plus} and Theorem \ref{resplus} are contained in Sections \ref{118+} and \ref{res+}, respectively. Theorem \ref{fhnew} is then proved in Section \ref{fhn}.

\subsection{Acknowledgments} The author thanks his colleagues Scott Ahlgren and Alexandru Zaharescu for answering his naive number theory questions.  He would also like to thank Viktor Ginzburg for useful discussions.

\section{Preliminaries}

\subsection{Number theoretic tools}
Let $p_j$ denote the $j$th prime number. By convention, we will set $p_0=1$ although we will refrain from calling $1$ a prime number. Set 
\begin{equation*}
\label{ }
(x) = x-\lfloor x \rfloor.  
\end{equation*}
The following foundational result by I.M. Vinogradov underlies our proof of Theorem \ref{118plus}.
\begin{Theorem}(\cite{vin})
\label{prime}
If $\vartheta$ is irrational then the sequence $ (p_j \vartheta)$ is uniformly distributed on $(0,1)$.
\end{Theorem}

To prove Theorem \ref{resplus} we will also require a multidimensional generalization of Theorem \ref{prime} due to G. Harman,  \cite{ha}. In fact, the following immediate implication of Theorem 1.4 of \cite{ha} is sufficient for our purposes.

\begin{Theorem}\label{kprime}(\cite{ha})
If the numbers $1, \vartheta_1, \dots \vartheta_k$ are rationally independent then the sequence
$$ 
\left((p_j \vartheta_1), \dots , (p_j \vartheta_k) \right)
$$
is dense in $(0,1)^k$.
\end{Theorem}

\subsection{Symplectic tools}
In what follows $(M, \omega)$ will be a closed  symplectic manifold of dimension $2n$ whose minimal Chern number will be denoted by $N$. We will assume that $(M, \omega)$ is weakly monotone, that is $N \geq n-2$
or  $[\omega]|_{\pi_2(M)} = \lambda c_1(M)|_{\pi_2(M)}$ for some nonnegative constant $\lambda$.
We will also assume that $(M, \omega)$ is rational, that is $\langle [\omega], \pi_2(M) \rangle  = \lambda_0 \mathbb{Z}$ for some positive constant $\lambda_0$ which we will refer to as the index of rationality.  When 
$\langle [\omega], \pi_2(M) \rangle  = 0$ we set $\lambda_0 =\infty$.

\subsubsection{Symplectic isotopies}
Let $\psi_t$  be a smooth isotopy of symplectic diffeomorphisms of $(M, \omega)$, where $t$ takes values in $[0,1]$ and $\psi_0$ is the identity map. 
Denote the set of fixed points of $\psi_1$  by $\Fix(\psi_1)$, 
and the set of periodic points by $\Per(\psi_1).$ 
The period of  a point $P \in \Per(\psi_1)$ is defined to be the smallest positive integer $k$ for which $P \in \Fix((\psi_1)^k)$. 

We will assume throughout that the generating vector fields of our symplectic isotopies $\psi_t$ are all time-periodic with period one. This assumption  imposes no new restrictions on the time-one maps considered as any symplectic isotopy is homotopic, relative its endpoints, to one with this property. It allows us to 
consider the path $\psi_t$ for all $t \in \R$ and to identify $\psi_{k}$ with $(\psi_1)^k$.

The subset of symplectic isotopies we are most interested in are  Hamiltonian flows.
 A Hamiltonian on $(M, \omega)$ is a function $H\colon \R/\Z\times M\to \R$, or equivalently 
a smooth one-periodic family of functions $H_t(\cdot) = H(t, \cdot)$. Each such Hamiltonian determines a one-periodic
 vector field $X_H$ on $M$, via the equation
$i_{X_H} \omega = -dH_t$, whose time-$t$ flow will be denoted by $\phi^t_H$. This flow is defined for all $t \in \R$ and defines  a smooth isotopy of symplectic diffeomorphisms. The set of   time one maps $\phi=\phi^1_H$ 
of Hamiltonian flows constitutes the group Hamiltonian diffeomorphisms of $(M, \omega)$.

\subsubsection{The Conley-Zehnder and mean indices}\label{meanindex}

Consider a smooth isotopy $\psi_t$ of symplectic diffeomorphisms as above. For $P \in \Fix(\psi_1)$,  let  $x \colon [0,1] \to M$
be the closed curve $\psi_t(P)$. Given a symplectic trivialization $\xi$ of $x^*TM$, the linearized flow of $\psi_t$ along $x(t)$ yields a smooth path $A_{\xi}\colon [0,1] \to Sp(n)$ of symplectic matrices 
starting at the  identity matrix. One can associate to $A_\xi$ its Conley-Zehnder index $\mu(A_\xi) \in \Z$ 
as defined in \cite{cz3}, and its  mean index $\Delta(A_\xi) \in \R$ as defined in \cite{SZ}. These quantities depend only on the homotopy class of the symplectic trivialization $\xi$. We denote this class by $[\xi]$
and define the Conley-Zehnder and mean index of $P$ with respect to this choice as
\begin{equation*}
\label{ }
\mu(P;\psi_t, [\xi]) = \mu(A_{\xi})
\end{equation*} 
 and 
\begin{equation*}
\label{ }
\Delta(P;\psi_t, [\xi]) = \Delta(A_{\xi}),
\end{equation*}
respectively. The following properties of these indices are proved in \cite{SZ}. 

\medskip
\noindent {\bf Approximation.} The mean index is never too far from the Conley-Zehnder index as
\begin{equation}
\label{closer}
|\mu(P;\psi_t, [\xi]) - \Delta(P;\psi_t, [\xi])| \leq n,
\end{equation}
where the strict form of the inequality holds if the linearization of $\psi_1$ at $P$ has at least one eigenvalue different from $1$.

\medskip
\noindent {\bf Iteration formula.} The mean index grows linearly under iteration, i.e., 
\begin{equation}
\label{iter}
\Delta(P; \psi_{tk}, [\xi^k]) = k \Delta(P;\psi_t, [\xi]),
\end{equation}
where $\xi^k$ is the trivialization of $TM$ along $\psi_{tk}(X)$ induced by $\xi$.

\medskip
\noindent {\bf Continuity.} The mean index is also continuous with respect to $C^1$-small perturbations of the symplectic isotopy, \cite{SZ}. To state this property precisely, we first note that if two fixed points $P$ and $P'$, of possibly different maps $\psi_1$ and $\psi'_1$,  represent the same homotopy class $c \in \pi_1(M)$ then  any  choice of $[\xi]$ for $P$ determines a unique class of symplectic trivializations for $P'$ which we still denote by $[\xi]$. When we compare indices of fixed points  in the same homology class we will always assume that the classes of trivializations being used  are coupled in this manner.

Now let $\widetilde{\psi}_t$ be a symplectic isotopy $C^1$-close $\psi_t$. 
Under this perturbation, each fixed point $P$ of $\psi_1$ splits into a collection of fixed points of $\widetilde{\psi}_1$ which are close to $P$ (and hence in the same homotopy class as $P$). 
If $\widetilde{P}$ is one of these  fixed points of $\widetilde{\psi}_1$ then $$|\Delta(P; \psi_t, [\xi]) - \Delta(\widetilde{P}; \widetilde{\psi}_t, [\xi])|$$ is small. 

\medskip

Finally we recall that  in dimension two there is, generically,  a very simple relation between the mean and Conley-Zehnder indices. The following result can be easily derived, for example,  from Theorem 7 in Chapter 8 of \cite{lo}. 

\begin{Lemma}\label{determine}
 Let $(M,\omega)$ be a  two-dimensional symplectic manifold  and suppose that  $\psi_t$ is an isotopy of symplectic diffeomorphisms of $(M,\omega)$ starting at the identity. If $P$ is a 
fixed point of $\psi_1$ and $\Delta(P;\psi_t, [\xi])$ is not an integer, then $\mu(P;\psi_t, [\xi])$ is the odd integer closest to $\Delta(P;\psi_t, [\xi])$. 
\end{Lemma}

\subsubsection{Indices of contractible fixed points modulo $2N$.} When $P$ is a contractible fixed point of $\psi_t$, that is $x(t)=\psi_t(P)$ is contractible,
it is often useful to restrict attention to  trivializations of $x^*TM$ determined by a choice of smooth  spanning disc $u \colon \mathbb{D}^2 \to M$
with $u(e^{2\pi i t}) = x(t)$. For such choices of trivializations the corresponding  indices are well-defined
modulo twice the minimal Chern number, $2N$. In fact, the corresponding elements of $\R/2N\Z$ depend only on the time one map $\psi_1$ and hence will be denoted by $\mu(P)$ and $\Delta(P)$. 

\subsubsection{Actions of contractible fixed points modulo $\lambda_0$}
Let $P$ be a contractible fixed point of a Hamiltonian diffeomorphism $\phi^1_H$ and  set $x(t)= \phi^t_H(P)$.
If $u \colon \mathbb{D}^2 \to M$ is  a spanning disc for $x(t)$ as above, the action of $P$ with respect to $u$  is defined  to be
\begin{equation}
\label{ }
\Aa_H(P,u) = \int_0^1H(t, x(t))\, dt - \int_{\mathbb{D}^2} u^* \omega.
\end{equation}
For a different choice of spanning disc $v$ we have $\Aa_H(P,u) - \Aa_H(P,v) \in \lambda_0 \Z$.  
The quantity $\Aa_H(P) = \Aa_H(P,u) \mod \lambda_0$ is then well defined. This is what we will call the action of $P$ with respect to $H$. The action spectrum of $H$ is defined here to be the set 
\begin{equation}
\label{ }
\Ss(H) =\{ \Aa_H(P) \mid P \in \Fix(\phi^1_H), \text{   $P$ is contractible} \} \subset \mathbb{R}/  \lambda_0 \Z.
\end{equation}

The $k$th iteration of $\phi^1_H$, $\phi^k_H$, is generated by the Hamiltonian 
\begin{equation}
\label{ }
H_k(t,p) = kH(kt,p).
\end{equation} More precisely, we have $\phi^t_{H_{k}} = \phi^{kt}_H$ for all $t \in \R$. 
The action, like the mean index grows linearly with iteration, in the sense that for all $k \in \N$ and $P \in \Fix(\phi^1_H)$ we have
\begin{equation}
\label{ }
\Aa_{H_k}(P) = k \Aa_{H}(P) \in \mathbb{R}/  \lambda_0 \Z.
\end{equation}
In particular, 
\begin{equation}
\label{ }
k\Ss(H) \subset \Ss(H_k).
\end{equation}

\subsubsection{Hamiltonian Floer homology} Let $H$ be a Hamiltonian on $(M, \omega)$. The contractible $1$-periodic orbits of 
of $X_H$ are in one-to-one correspondence with the contractible fixed points  of $\phi^1_H$. Assuming these are all nondegenerate they generate the Floer complex of $H$ and  the corresponding Hamiltonian Floer homology of $H$, $\HF(H)$. This   is isomorphic to $\mathrm{HQ}(M)$, the quantum homology of $M$. More precisely, with respect to the standard gradings of  $\HF(H)$ by the Conley-Zehnder index,  we have $\HF_*(H) \cong \mathrm{HQ}_{*+n}(M)$.

We will also require two related notions. The first of these, introduced by Floer and Hofer in \cite{fh},  is $\HF^{(a,b)}(H)$, the Hamiltonian Floer homology of $H$ restricted to the action window $(a,b)$. 
The second required variant of Floer homology needed is that  of local Floer homology as introduced by Floer in \cite{fl}. In the present setting, the local Floer homology is a group $\HF^{\loc}_*(H, P, u)$ associated to an isolated contractible fixed point $P$ of $\phi^1_H$ and a spanning disc  $u \colon \mathbb{D}^2 \to M$ for the smooth loop $x(t)= \phi^t_H(P)$. The only detail of the construction of local Floer homology relevant here is that the first step is to perturb $H$, if necessary, so that $P$ breaks into a collection nearby fixed points which are nondegenerate.

\subsubsection{Floer homology of symplectic diffeomorphisms}\label{floerhomology}
Finally, we will  need to consider  the Floer homology of symplectic diffeomorphisms. In fact, we need only consider the very special case when $(M,\omega)$ is a two-dimensional symplectic torus $(\T^2, \Omega)$ and the symplectic diffeomorphism is isotopic to the identity. The relevant details of this construction are described briefly below and the reader is referred  to  \cite{ds, se1, se2, se3, va} for more details on the general construction,  and to \cite{c1,c2}  for more thorough reviews of the Floer theory of symplectic diffeomorphisms of surfaces.

Consider a smooth isotopy  $\psi_t$ of symplectic diffeomorphisms of $(\T^2, \Omega)$ starting at the identity such that the fixed points of $\psi_1$ are all nondegenerate.  The Floer homology of $\psi_1$, $\FH(\psi_1)$, is then a well-defined group with the following properties.

\medskip

\noindent{\bf Invariance under Hamiltonian isotopy:} If $\phi$ is a Hamiltonian diffeomorphism of  $(\T^2, \Omega)$
  then \begin{equation*}
\label{ }
\FH(\psi_1 \phi)= \FH(\psi_1).
\end{equation*}
  
  \medskip

\noindent{\bf Splitting:} The Floer homology $\FH(\psi_1)$ admits a decomposition of the form 
\begin{equation*}
\label{ }
 \FH(\psi_1) = \bigoplus_{c \in \mathrm{H}_1(\T^2;\Z)} \FH(\psi_1; c).
\end{equation*}
Here, each summand $\FH(\psi_1; c)$ is the homology of a chain complex
$(\CF(\psi_1; c), \partial_J)$ where the chain group $\CF(\psi_1; c)$ is a torsion-free module over a suitable Novikov ring, and the rank of this module is the number of  fixed points  of $\psi_1$ which represent the class $c$.  The group $\FH(\psi_1; 0)$ coincides with the Floer-Novikov Homology constructed by L{\^e} and Ono in \cite{va}, (\cite{se3}).  Moreover, if $\psi_t$ is homotopic to a Hamiltonian isotopy $\phi^t_H$, then $\FH(\psi_1) = \FH(\psi_1; 0)= \HF(H)$, \cite{fl}. In particular, $\FH(\psi_1)$ is canonically isomorphic to $\mathrm{H}(\T^2;\Z)$.
 \medskip

\noindent{\bf Grading:}  Each chain complex $(\CF(\psi_1; c), \partial_J)$ above has a relative $\Z$-grading and the boundary operator decreases degrees by one. For the case $c=0$, the grading 
can be set by using the usual Conley-Zehnder index  of contractible fixed points (which is well-defined since $c_1(\T^2) = 0$.) For example   if $\psi_t$ is homotopic to a Hamiltonian isotopy we have  
\begin{equation}
\label{hfham}
\FH_*(\psi_1; 0) = \mathrm{H}_{*+1}(\T^2;\Z).
\end{equation}
For a general class $c \in  \mathrm{H}_1(\T^2;\Z)$ the (relative) grading of $(\CF_*(\psi_1; c), \partial_J)$ is again determined by the Conley-Zehnder index and the overall shift can be fixed by choosing a homotopy class of symplectic trivializations of $z^*(T\T^2)$ where $z \colon S^1 \to \T^2$ is a smooth representative of $c$.

\medskip

\noindent{\bf Extension to all smooth isotopies:}
The property of invariance under Hamiltonian isotopy allows one to also define the Floer homology for any  smooth symplectic isotopy $\widetilde{\psi}_t$ of $(\T^2, \Omega).$ 
One simply  sets 
\begin{equation*}
\label{ }
\FH_*(\widetilde{\psi}_1) = \FH_*(\widetilde{\psi}_1 \circ \phi) 
\end{equation*}
where $\phi$ is a Hamiltonian diffeomorphism for which the fixed points of  $\widetilde{\psi}_1 \circ \phi$ are nondegenerate. For example, if $\widetilde{\psi}_t = id$ for all $t \in [0,1]$ we can perturb by the Hamiltonian flow of a $C^2$-small Morse function to
obtain
\begin{equation}
\label{hfidentity}
\FH(id) = \FH(id;0) = \mathrm{H}(\T^2;\Z).
\end{equation}

\medskip

\noindent{\bf Dichotomy:} The following alternative is well known and plays a crucial role in the proof of Theorem \ref{fhnew}. A proof of this alternative can be found in \cite{res1}, for example.

\begin{Proposition}\label{hom}
Either $\psi_1$ is a Hamiltonian diffeomorphism in which case $\FH(\psi_1) = \FH(\psi_1; 0)$ and
 $\FH_*(\psi_1; 0)= \mathrm{H}_{*+1}(\T^2;\Z)$, or 
$\FH(\psi_1)$ is trivial.
\end{Proposition}

\section{Vinogradov's Hanukkah gift: Symplectically degenerate maxima and periodic orbits of prime period}\label{118+}

Let $ \varphi= \phi^1_H$ be a Hamiltonain diffeomorphism of  $(M, \omega)$ and let $P$ be a fixed point of $\varphi$ such that the loop $x(t) =\phi^t_H(P)$ is contractible. The point $P$ is called a 
{\it symplectically degenerate maximum} of $\varphi$ if for some spanning disc $u \colon \mathbb{D} \to M$ of $x$ we have 
\begin{equation}
\label{max}
\Delta(P; \phi^t_H, [u] ) = 0  \text{   and   } \HF^{\loc}_n(H, P,u) \neq 0,
\end{equation}
where $[u]$ denotes the equivalence class of symplectic trivializations of $TM$ along $x$ determined by $u$.
The following result encapsulates the breakthrough at the heart of the proof of the 
Conley conjecture by Hingston in \cite{Hi} and Ginzburg in \cite{Gi:conley}. It is established in this more general setting  by Ginzburg and G\"urel in \cite{gg2}.

\begin{Theorem} (Theorem 1.17 of \cite{gg2})
\label{117}
Suppose $(M, \omega)$ is weakly-monotone and rational  and that $P$ is a symplectically degenerate maximum of $\varphi = \phi^1_H$.  Let $u$ be a spanning disc as in \eqref{max} and set $c= \Aa_H(P,u)$. For every sufficiently small $\epsilon>0$ there is a $k_{\epsilon}>0$ such that 
\begin{equation}
\label{ }
\HF_{n+1}^{(kc+\delta_k, kc + \epsilon)}(H_k) \neq 0
\end{equation}
for all $k>k_{\epsilon}$ and some $\delta_k$ with $0<\delta_k < \epsilon.$
\end{Theorem}
 
Also proved in  \cite{gg2} is the following result concerning the periodic points of a Hamiltonian diffeomorphism with 
a symplectically degenerate maximum.

\begin{Theorem} (part (i) of Theorem 1.18 of \cite{gg2})
\label{118}
Suppose $(M, \omega)$ is weakly-monotone and rational. If $\varphi$ has finitely many contractible fixed points and one of these is a symplectically degenerate maximum, then $\varphi$ has infinitely many geometricially distinct  contractible periodic points.
\end{Theorem}

In this section we prove Theorem \ref{118plus} which improves Theorem \ref{118} by detecting, under the same hypotheses,  infinitely many contractible  periodic points of arbitrarily large prime period.

\subsection{Proof of Theorem \ref{118plus}}
In fact, we only need to alter the proof from \cite{gg2} by considering prime iterates and replacing the role of the standard equidistribution theorem with Vinogradov's prime distribution theorem. Because it is short, we include the entire argument  for completeness, following 
very closely the presentation from \cite{gg2}.

Arguing by contradiction we assume that  $\varphi=\phi^1_H$ does not have contractible  periodic points of arbitrarily large prime period and  derive a contradiction.  Let $P$ be a symplectically degenerate maximum of $\phi^1_H$ and $u$ a spanning disc of $\phi^t_H(P)$ such that \eqref{max}  holds. For convenience, we will  normalize (add the appropriate constant to) our generating Hamiltonian $H$ so that $\Aa_H(P,u) =0$.

If $\phi^1_H$ does not have periodic points of  prime period set $j_0=0$. Otherwise, let $j_0>0$ be the smallest natural number such that $p_{j_0}$ is greater than the
largest prime period of $\phi^1_H$. We then have
\begin{equation}
\label{fixn}
\Fix(\phi_H^{p_j}) = \Fix(\phi^1_H)
\end{equation} 
for all $j\geq j_0$.

Now, the action spectrum of $H$, $\Ss(H)$ can be divided into action values which are rational and irrational modulo the index of rationality, $\lambda_0$, i.e., 
\begin{equation}
\label{ }
\Ss(H) = \lambda_0 \{\eta_1, \dots, \eta_r \} \cup \lambda_0 \{m_1/n_1, \dots , m_s/n_s \}
\end{equation}
where the $\eta_i$ are irrational and the $m_i$ and $n_i$ are integers.
From \eqref{fixn} we also have  
\begin{equation}
\label{actionn}
\Ss(H_{p_j})  =p_j \Ss_{\lambda_0}(H) 
\end{equation}
for all $j\geq j_0$ where, again, the flow of the Hamiltonian $H_{p_j}$ is $\phi^{tp_j}_H$.

Let $d$ be the least common multiple of the $n_i$. Choose the $\epsilon$ in Theorem \ref{117} to be less than $1/d$ and let $k_{\epsilon}$ be the corresponding integer promised in that theorem.
Choose $j_1\geq j_0$ such that $p_j> k_{\epsilon}$ for all $j  \geq j_1$. Theorem \ref{117} implies that for all primes $p_j >k_\epsilon$ one of the elements of  $\Ss(H_{p_j})$
lies in the interval $(0, \epsilon) \subset \R/\lambda_0\Z $. For all $j \geq j_1$ it follows from \eqref{actionn} and our choice of $\epsilon$ that this element  of  $\Ss(H_{p_j})$ in $(0, \epsilon)$ must be of the form $p_j\eta_i$ for some $i=1, \dots, r$.  In conclusion, under our assumption on $\varphi$, for all sufficiently small $\epsilon>0$ there is a $j_1 \geq j_0$ such that for every $j\geq j_1$ there is an integer $i \in [1,r]$ such that  $p_j\eta_i \in (0, \epsilon)$.
We now derive the desired contradiction by showing that this statement is false.

For $\ell>j_1$, set 
\begin{equation}
\label{ }
\Pp_i(\ell, \epsilon) = \frac{|\{p_{j_1+1}\eta_i , \dots, p_\ell \eta_i\} \cap (0, \epsilon)|}{\ell-j_1},
\end{equation}
In other words,  
$\Pp_i(\ell, \epsilon)$ is the probability that  $p_j\eta_i$ lies in $(0, \epsilon)$ for $j \in (j_1, \ell]$. 
By Theorem \ref{prime} we have 
$$
\lim_{\ell \to \infty} \Pp_i(\ell, \epsilon) = \frac{\epsilon }{\lambda_0}.
$$
Hence
$$
\lim_{\ell \to \infty} \sum_1^r \Pp_i(\ell, \epsilon) = \frac{r \epsilon }{\lambda_0}
$$
and for or $\epsilon < \lambda_0/r$ and $\ell$ sufficiently large, we have  $\sum_1^r \Pp_i(\ell, \epsilon)<1$. It follows that for such an $\epsilon$ and $\ell$  there must be a $j \in (j_1, \ell]$ such that $p_j\eta_i \notin (0, \epsilon)$ for all $i =1, \dots, r$.  This contradicts the statement above and concludes the proof.

\section{Resonance relations for fixed points of Hamiltonian diffeomorphisms}\label{res+}

Let  $\varphi$ be a Hamiltonian diffeomorphism of $(M, \omega)$ which has finitely many contractible fixed points. 
Let $\Delta_1,\ldots, \Delta_m$ be the collection of  nonzero mean indices 
of the contractible fixed points of $\varphi$.  A  \emph{resonance relation}
for $\Fix(\varphi)$ is a vector $\va=(a_1,\ldots,a_m)\in\Z^m$ such that
$$
a_1\Delta_1+\ldots+a_m\Delta_m=0\mod 2N.
$$
The resonance relations of $\Fix(\varphi)$ form a free abelian group $\CR(\varphi)\subset \Z^m$.  

In this section we use Theorem \ref{118plus} and  Theorem \ref{kprime} to prove Theorem \ref{resplus}, which we 
restate here for convenience. 

\begin{Theorem}
Suppose that $(M, \omega)$ is weakly-monotone and rational and the minimial Chern number $N$ is finite and greater  than $n$. 
Let $\varphi$ be a Hamiltonian diffeomorphism which has finitely many contractible fixed points
and does not have contractible fixed points of arbitrarily large period in $\mathbb{P}^2$.
Then $\CR(\varphi) \neq 0$, i.e., the nonzero mean indices $\Delta_i$ satisfy at least
 one non-trivial resonance relation.
\end{Theorem}

\noindent Again $\mathbb{P}^2$ is the set of natural numbers greater than one which have at most  two prime factors when counted with multiplicity.
\subsection{Proof of Theorem \ref{resplus}}
Divide the set of contractible fixed points of $\varphi$ into two subsets;  the points $\{P_1, \dots, P_m\}$  with nonzero mean indices (modulo $2N$), and the points  $\{Q_1, \cdots , Q_l\}$ with zero mean indices (modulo $2N$).  Set $\Delta_j = \Delta(P_j)$ and consider the $m$-tuple of nonzero mean indices  $$\Delta(\varphi) = (\Delta_1,
\ldots, \Delta_m) \in (\R/2N\Z)^{m}.$$ Define $\Pi$ to be the open cube $(2n,2N)^{m}$ in
the torus  $(\R/2N\Z)^{m}$.  

\begin{Proposition}\label{notdense}
There is a $j_0\geq 0$ such that for any $j  \geq j_0$ the 
point $p_j \Delta(\varphi) =  (p_j \Delta_1, \ldots, p_j \Delta_m) \in (\R/2N\Z)^{m}$ is not contained
in $\Pi$.
\end{Proposition}

\begin{proof} If $\varphi$ does not have periodic points with period in $\mathbb{P}^2$ set $j_0=0$. Otherwise, let $j_0>0$ be the smallest natural number such that $p_{j_0}$ is greater than the
largest  period of $\varphi$ in $\mathbb{P}^2$.
We then have 
\begin{equation}
\label{identify}
\Fix(\varphi^{p_j}) = \Fix (\varphi)  \text{  for all  } j \geq j_0.
\end{equation}

Choose a Hamiltonian $H$ with $\varphi = \phi^1_H$. The fixed point set and the mean indices (modulo $2N$) are independent of this choice.  For any Hamiltonian $G$ on $(M, \omega)$, the Hamiltonian Floer homology, $\HF_*(G)$, is graded  modulo $2N$ by the Conley-Zehnder index and is isomorphic to $\mathrm{HQ}_{n+*}(M)$. Since $N>n$, the nontrivial {\it fundamental class} $\HF_n(G)$ is distinguished by its degree.  In particular, for all  $p_j$ we  have 
 \begin{equation}
\label{neq}
\HF_n(H_{p_j}) \neq 0. 
\end{equation}
For this to hold  there must be a contractible fixed point $X$ of $\phi^1_{H_{p_j}}$ and a spanning disc $u$ of $\phi^t_{H_{p_j}}(X)$ such that the local Floer homology $ \HF^{\loc}_n(H_{p_j}, X,u)$ is nontrivial. For such an $X$ it follows from the construction of local Floer homology, the continuity properties of the mean index and   \eqref{closer} that 
\begin{equation}
\label{ber}
\Delta(X; \phi^t_{H_{p_j}}, [u]) \in [0, 2n].
\end{equation}
By \eqref{identify}, $X$ is a fixed point of $\varphi =\phi^1_H$ and so \eqref{iter} implies that 
\begin{equation}
\label{mult}
\Delta(X; \phi^t_{H_{p_j}}, [u])=p_j\Delta(X; \phi^t_H, [v]) \mod 2N
\end{equation}
for any choice of spanning disc $v$.
We will therefore be done if we can prove that $X=P_i$ for some $i \in [1,m]$. For if $X=P_i$ then it follows from \eqref{mult} that  the $i$th component of $p_j \Delta(\varphi)$  is equal to $$\Delta(P_i; \phi^t_{H_{p_j}}, [u])\mod 2N$$ and \eqref{ber} then implies that this component lies  in $[0,2n] \subset \R /2N\Z$. Hence,  $p_j \Delta(\varphi)$ is not contained in $\Pi$. 

Arguing by contradiction, assume instead that $X=Q_i$ for some $i \in [1,l]$. Then $$\Delta(Q_i;\phi^t_{H_{p_j}}, [u])= p_j\Delta(Q_i,  \phi^t_H, [v]) =0 \mod 2N.$$ Since  $N>n$, it then follows from \eqref{ber} that $\Delta(Q_i;\phi^t_{H_{p_j}}, [u]) =0$.
We also have   $ \HF^{\loc}_n(H_{p_j}, Q_i,u) \neq 0$. So, the point $Q_i$ is a symplectically degenerate maximum of $\phi^1_{H_{p_j}} = \varphi^{p_j}$. Theorem \ref{118plus}  then implies that $\varphi^{p_j}$ has contractible periodic points of arbitrarily large prime period. This contradicts the hypothesis of Theorem \ref{resplus} that $\varphi$ does not have contractible fixed points of arbitrarily large period in $\mathbb{P}^2$. 
\end{proof}

To conclude, we note that Proposition \ref{notdense} implies that the set 
$$
\{\left((p_j\Delta_1), \dots, (p_j\Delta_m)\right) \mid j\geq j_0\}
$$
is not dense in $(0,2N)^m$.  It then follows from Theorem \ref{kprime} that the numbers 
$1, \Delta_1,\dots, \Delta_m$ must  be rationally dependent. Hence, the rank 
of $\mathcal{R}(\varphi)$ is at least one and  the proof of Theorem \ref{resplus} is complete.

\section{Period growth for Hamiltonian diffeomorphisms of the sphere}\label{fhn}
We now prove Theorem \ref{fhnew} which we again restate for convenience.
\begin{Theorem}
Let $\phi$ be a Hamiltonian diffeomorphism of $S^2$ with at least three fixed points such that 
 $\Fix(\phi^{k})$ is finite for all $k \in \N$. Either $\phi$ has infinitely many periodic points with arbitrarily large period and rational mean indices not equal to zero modulo four, or there exists a $K \in \N$ such that $\phi^K$ has infinitely many periodic points with arbitrarily large periods in $\mathbb{P}^2$.
\end{Theorem}

\subsection{Sorting and a reduction to two cases}
We begin by describing  a recursively defined sorting procedure. First we sort the fixed points of $\phi$ into  three disjoint subsets defined in terms of their mean indices;
\begin{equation}
\label{ }
\Fix(\phi) = \mathcal{Z}_0 \cup \Ii_0 \cup \mathcal{Q}_0 
\end{equation}
where $P$ is in $\mathcal{Z}_0$ if 
$$\Delta(P)=0 \mod 4,$$  $P$ is in $\Ii_0$ if  $\Delta(P)$ is irrational 
and $P$ is in $\mathcal{Q}_0$ otherwise. 
If $\mathcal{Q}_0 = \emptyset$ we stop. Otherwise we set
$$K_1 = \min\{ k \in \N \mid k\Delta(P) =  0 \mod 4 \text{  for all  }  P \in \mathcal{Q}_0\},$$
and then  consider the map $\phi^{K_1}$. Sorting its fixed points into three disjoint sets as before, we get
\begin{equation}
\label{ }
\Fix(\phi^{K_1}) = \mathcal{Z}_1 \cup \Ii_1 \cup \mathcal{Q}_1.
\end{equation}
Clearly,  $\Fix(\phi) \subset \Fix(\phi^{K_1})$. In particular, it follows from the iteration formula for the mean index 
and our choice of $K_1$ that 
$\mathcal{Z}_0 \cup \mathcal{Q}_0 \subset \mathcal{Z}_1$, $\Ii_0 \subset \Ii_1$ and any element 
of $\mathcal{Q}_1$ is a periodic point of $\phi$ whose  period lies in $(1,K_1]$ and divides $K_1$.  If $\mathcal{Q}_1 = \emptyset$
we stop otherwise we define $K_2$ as above, and proceed in the same manner. 
Either we can repeat this procedure forever or it terminates at some first step, say the $N$th, at which $Q_N$ is empty.
In the latter case there must be a sequence of periodic points $Q_i \in \mathcal{Q}_i$ whose mean indices are rational and not equal to zero modulo four, such that the period of $Q_i$ is greater than $\prod_{j=1}^{i-1}K_j \geq 2^{i-1}$. This corresponds to the first  period growth alternative of Theorem \ref{fhnew}. 

It  remains to prove that if $Q_N = \emptyset$ for some finite $N$ then for some $K \in \N$ the map $\phi^K$
has periodic points with arbitrarily large periods in $\mathbb{P}^2$. Arguing by contradiction, we set 
 $\varphi = \phi^{K_1 \cdots K_N }$ and assume that $\varphi$ does not have infinitely many periodic points with arbitrarily large periods in $\mathbb{P}^2$.

By our choice of $\varphi$, we have  
\begin{equation}
\label{ }
\Fix(\varphi) = \mathcal{Z} \cup \Ii
\end{equation}
where the points in $\mathcal{Z}$ have zero mean index modulo four, and the points in $\Ii$ have irrational mean indices.

\begin{Lemma}\label{two}
The set $\Ii$ has at least two elements.
\end{Lemma}

\begin{proof}
By  assumption, Theorem \ref{resplus} applies to $\varphi$ and implies that  the mean indices of the elements of $\Ii$ must satisfy at least one nontrivial resonance relation.
Since these mean indices are irrational there are at least two elements of $\Ii$. 
\end{proof}

\begin{Lemma}
The set $\mathcal{Z}$ is not empty.
\end{Lemma}

\begin{proof}
Arguing by contradiction assume that $\Fix(\varphi) = \Ii$. Since $\Fix(\phi) \subset \Fix(\varphi)$ it follows from the hypotheses of Theorem \ref{fhnew} that there must be at least 
three elements of $\Fix(\varphi)$.  Since they have irrational mean indices, each of these fixed points is nondegenerate  and elliptic and so their  topological indices are all even. By the Lefschetz fixed point theorem we then have the Euler characteristic of $S^2$ equal to $|\Fix(\varphi)| >2$, a contradiction.
\end{proof}

Following \cite{res1} we now separate the proof into two cases; the one in which at least 
one point of $\mathcal{Z}$  is degenerate, and the other in which every element of $\mathcal{Z}$ is nondegenerate.  
In both cases we will derive a contradiction to Proposition \ref{hom}.

\subsection{Case 1: at least 
one point in $\mathcal{Z}$ is degenerate}
 
Assume that $\varphi$ as above has a fixed point, say  $W$, which is degenerate and satisfies $\Delta(W) = 0 \mod 4$. 
By Lemma \ref{two} we can also choose two elements $X$ and $Y$ of $\Ii$.
\subsubsection{Step 1: A useful generating Hamiltonian for $\varphi$}\label{staticham} We  first  choose a generating Hamiltonian $H \colon \R/\Z \times S^2 \to \R$ for $\varphi$ such that  $W$ and $X$ are fixed points of $\phi^t_H$ for all $t \in \R$. The construction of such a Hamiltonian is described in \S 3.3.1 of \cite{res1}.
%We may also assume that 
%$H$ vanishes when $t$ is within some small fixed distance, say $0<\delta_H \ll 1$, of  $0\in \R/\Z$. 

For each $k \in \N$, the Hamiltonian diffeomorphism $\varphi^k$ is then generated by the Hamiltonian 
\begin{equation*}
\label{ }
H_{k}(t,P)=kH(kt, P)
\end{equation*}
and $\phi^t_{H_{k}} = \phi^{kt}_H$ for all $t \in \R$. So, $W$ and $X$ are still static fixed points of 
the flow of $H_{k}$. 

\subsubsection{Step 2: A useful perturbation} Under our assumption that 
$\varphi$ does not have periodic points of arbitrarily large period in $\mathbb{P}^2$, there is a $j_0 \in \N$ such that  
\begin{equation}
\label{ }
\Fix(\varphi^{p_j^2}) = \Fix{\varphi}
\end{equation}
for all $j \geq j_0$. For each such $j$ we now describe a perturbation of the time-one flow of the Hamiltonian $H_{p_j^2}$.

\begin{Lemma}(Lemma 3.4, \cite{res1})\label{k-pert}
For each $j \geq j_0$ there is a neighborhood $U_j$ of $W$ and a Hamiltonian flow $\widetilde{\phi}_{j,t}$ which
 is arbitrarily $C^{\infty}$-close to $\phi^t_{H_{p_j^2}}$, is equal to $\phi^t_{H_{p_j^2}}$ outside of $U_j$, and whose fixed point set has the form
 \begin{equation*}
\label{ }
\Fix(\widetilde{\phi}_{j,1}) =  \Fix(\varphi) \cup \{ W_1, \dots, W_d\},
\end{equation*}
where
\begin{itemize}
  \item[(i)]  $W$ is a fixed point of $\widetilde{\phi}_{j,t}$ for all $t$, an elliptic fixed point of $\widetilde{\phi}_{j,1}$, and 
\begin{equation}
\label{ }
\Delta(W; \widetilde{\phi}_{j,t}, [\xi^{p_j^2}]) = p_j^2 \Delta(W; \phi^t_{H}, [\xi]) + \lambda/\pi
\end{equation}
where $[\xi]$ is any class of symplectic trivializations, and $\lambda>0$ is arbitrarily close to $0$. 
  \item[(ii)] the $W_i$ are all contained in $U_j$, they are nondegenerate, and $\Delta(W_i; \widetilde{\phi}_{j,t}, [\xi^{p_j^2}])$ is arbitrarily close to  $p_j^2\Delta(W; \phi^t_{H}, [\xi])$ for $i =1, \dots, d$ and any choice of $[\xi]$.
  \item[(iii)] none of the $\widetilde{\phi}_{j,t}$ trajectories of the remaining fixed points of $\widetilde{\phi}_{j,1}$ enter $U_j$. 
\end{itemize}

\end{Lemma}

\subsubsection{Step 3: transfer of dynamics to the torus}\label{complete}
We now construct  from $\widetilde{\phi}_{j,t}$ a symplectic isotopy $\psi_{j,t}$ of the torus.
We first blow-up the points $W$ and $X$ and complete the restriction of 
$\widetilde{\phi}_{j,1}$ to $S^2 \smallsetminus \{W,X\}$  to an area preserving diffeomorphism  of the closed annulis $[-1,1] \times \R/2\pi \Z$. The resulting map,  $\overline{\phi}_{j}$, acts on the boundary circles, $\Gamma_W = \{1\} \times  \R/2\pi \Z$ and  $\Gamma_X = \{-1\} \times  \R/2\pi \Z$, as the rotation by  $\lambda$ and $\pi \Delta(X; \widetilde{\phi}_{j,t}, [\xi])$, respectively,  for any choice of the class $[\xi]$.  Note that 
\begin{eqnarray*}
\Fix(\overline{\phi}_{j}) & = & ( \Fix(\varphi) \smallsetminus \{W,X\}) \cup \{ W_1, \dots, W_d\} \\
{} & = & ( \mathcal{Z} \smallsetminus W) \cup ( \Ii \smallsetminus X) \cup\{ W_1, \dots, W_d\}. 
\end{eqnarray*}

Following Arnold  (Appendix 9 of \cite{ar}), we now extend the map $\overline{\phi}_{j}$ to the torus  formed by gluing two copies of the domain annulus $[-1,1] \times \R/2\pi \Z$ along their common boundaries after  inserting two  connecting cylinders.  This allows us to extend the  map $\overline{\phi}_{j}$  to an area preserving map  $\psi_j$ of  $(\T^2, \Omega)$ which agrees with $\overline{\phi}_{j}$ on the two annuli, and is defined on the connecting cylinders so that  the overall map is smooth and  has no new fixed points. In particular,
$\Fix(\psi_j)$ consists of  two copies of $\Fix(\overline{\phi}_{j})$, which we denote by 
 \begin{equation*}
\label{ }
\Fix(\overline{\phi}^{\pm}_{j})=  ( \mathcal{Z} \smallsetminus W)^{\pm} \cup ( \Ii \smallsetminus X)^{\pm} \cup\{ W^{\pm}_1, \dots, W^{\pm}_d\}.
\end{equation*}
The isotopy $\widetilde{\phi}_{j,t}$ induces a smooth isotopy $\psi_{j,t}$ from the identity to $\psi_j$.

\subsubsection{The contradiction.} 
There are two fixed points of $\psi_j$, $Y^{\pm}$, corresponding to $Y$. The following result concerning the role of $Y^+$ in 
 the Floer complex generating the Floer homology $\FH(\psi_j)$ implies the desired contradiction.

\begin{Proposition}\label{not}
If  $j \geq j_0$ is sufficiently large  then $Y^+$ represents a nontrivial class in $\FH(\psi_j)$,  and if $Y^+$ is contractible then  the degree of the class $[Y^+]$ is  greater than one
in absolute value.
\end{Proposition}

\begin{proof}

Fix a class $[{\xi}]$ of symplectic trivializations of $TS^2$ along $\phi^t_{H}(Y)$. The resulting class $[{\xi}^{p_j^2}]$ of symplectic trivializations of $TS^2$ along $\phi^t_{H_{p_j}}(Q)$ then determines an equivalence class $[\xi^{p_j^2}_{+}]$ of symplectic trivializations of $T\T^2$ along $\psi_{j,t}(Y^+)$.
Let $P^{\pm}$ be any fixed point of $\psi_j$ in the same homotopy class as $Y^+$ where $P$ is the corresponding fixed point of $\widetilde{\phi}_{j,t}$. Since  the Floer boundary operator decreases degrees by one, to prove the first assertion of Proposition \ref{not} it suffices  to show that for $j$  large enough 
we have either 
\begin{equation*}
\label{ }
\Delta(P^{\pm}; \psi_{j,t}, [\xi^{p_j^2}_+]) = \Delta(Y^+; \psi_{j,t}, [\xi^{p_j^2}_+])
\end{equation*}
or 
\begin{equation*}
\label{ }
|\Delta(P^{\pm}; \psi_{j,t}, [\xi^{p_j^2}_+]) - \Delta(Y^+; \psi_{j,t}, [\xi^{p_j^2}_+])|>3.
\end{equation*}
In particular, by Lemma \ref{determine} the first equality implies $$\mu(P^{\pm}; \psi_{j,t}, [\xi^{p_j^2}_+]) = \mu(Y^+; \psi_{j,t}, [\xi^{p_j^2}_+]).$$ Whereas, if the alternative inequality holds and $\widetilde{P}^{\pm}$ is a nondegenerate fixed point obtained from $P^{\pm}$ by perturbation, then it follows from the continuity property of the mean index and \eqref{closer} that
\begin{equation*}
\label{ }
|\mu(\widetilde{P}^{\pm}; \psi_{j,t}, [\xi^{p_j^2}_+]) - \mu(Y^+; \psi_{j,t}, [\xi^{p_j^2}_+])|>1.
\end{equation*}
Since there are only finitely many points $P^{\pm}$, for any $j$,  it will follow that for large enough $j$ the point $Y^+$ is in the kernel of the corresponding Floer boundary map and is not in its image thus establishing the first assertion of Proposition \ref{not}.

\noindent{\it Subcase 1: $P$ corresponds to a fixed point of $\varphi^{p_j^2}$, i.e., $P \in (\mathcal{Z} \smallsetminus W) \cup (\Ii \smallsetminus X)$.}
By construction, we have 
$$
\Delta(Y^+; \psi_{j,t}, [\xi^{p_j^2}_+])  =\Delta(Y; \phi^t_{H_{p_j}}, [\xi^{p_j^2}])
$$
and the iteration formula \eqref{iter} then implies that 
\begin{equation}
\label{yplus}
\Delta(Y^+; \psi_{j,t}, [\xi^{p_j^2}_+])  =p^2_j\Delta(Y; \phi^t_H, [\xi]).
\end{equation}

Similarly, we have
\begin{equation}
\label{ppm}
\Delta(P^{\pm}; \psi_{j,t}, [\xi^{p_j^2}_+]) = p^2_j\Delta(P; \phi^t_H, [\xi]).
\end{equation}

If $\Delta(P; \phi^t_H, [\xi])=\Delta(Y; \phi^t_H, [\xi])$, these formulas imply that 
$$\Delta(P^{\pm}; \psi_{j,t}, [\xi^{p_j^2}_+])=\Delta(Y^+; \psi_{j,t}, [\xi^{p_j^2}_+]). $$

On the other hand, if $\Delta(P; \phi^t_H, [\xi])\neq\Delta(Y; \phi^t_H, [\xi])$ then 
equations \eqref{yplus} and \eqref{ppm} yield
\begin{equation*}
\label{ }
|\Delta(P^{\pm}; \psi_{j,t}, [\xi^{p_j^2}_+]) - \Delta(Y^+; \psi_{j,t}, [\xi^{p_j^2}_+])|  = p_j^2|\Delta(Y; \phi^t_H, [\xi]) - \Delta(P; \phi^t_H, [\xi])| .
\end{equation*} 
For sufficiently large $j \in \N$, we then have 
\begin{equation*}
\label{ }
|\Delta(P^{\pm}; \psi_{j,t}, [\xi^{p_j^2}_+]) -\Delta(Y^+; \psi_{j,t}, [\xi^{p_j^2}_+])| >3.
\end{equation*}

\noindent{\it Subcase 2: $P^{\pm}=W^{\pm}_j$.} 
Since $\Delta(W^{\pm}_j; \psi_{j,t}, [\xi^{p_j^2}_+]) = \Delta(W_j; \widetilde{\phi}_{j,t}, [{\xi}^{p_j^2}])$, it follows from
Lemma \ref{k-pert}  that $\Delta(W^{\pm}_j; \psi_{j,t}, [\xi^{p_j^2}_+])$ is arbitrarily close to
$p_j^2 \Delta(W; \phi^t_H, [{\xi}]).$
By assumption, $\Delta(W; \phi^t_H, [{\xi}])$ is an integer (multiple of four) and hence not equal to 
the irrational number $\Delta(Y; \phi^t_H, [\xi])$. Arguing again as in Case 1, we see that 
for sufficiently large $j$  
\begin{equation*}
\label{ }
|\Delta(W^{\pm}_j; \psi_{j,t}, [\xi^{p_j^2}_+]) - \Delta(Y^+; \psi_{j,t}, [\xi^{p_j^2}_+])| >3.
\end{equation*}

\medskip

Finally, we prove  the second assertion of Proposition \ref{not}. If $Y^+$ is a contractible fixed point  of $\psi_j$ then 
$\phi^t_H(Y)$ is contractible in $S^2 \smallsetminus \{W,X\}$. We choose $[\xi]$ in this case so that it is determined by a spanning disc for 
$\phi^t_H(Y)$. Then the induced class $[\xi^{p_j^2}_+]$ determines the canonical grading of $\FH_*(\psi_j)$. As described above, we have
$$
\Delta(Y^+; \psi_{j,t}, [\xi^{p_j^2}_+]) = p_j^2\Delta(Y; \phi^t_H, [\xi]).
$$
Since $\Delta(Y; \phi^t_H, [\xi])$ is irrational and hence nonzero,  we therefor have 
$$ |\Delta(Y^+; \psi_{j,t}, [\xi^{p_j^2}_+])| >2$$  for large enough $j \in \N$.
Inequality \eqref{closer} then yields $ |\mu(Y^+; \psi_{j,t}, [\xi^{p_j^2}_+])| >1.$

\end{proof}

Propositions \ref{not} and \ref{hom} can not both be true. The first assertion of Proposition \ref{not} together with 
Proposition \ref{hom} implies that $\psi_j$ must be a Hamiltonian diffeomorphism, in which case $\FH_m(\psi_j;0)$ 
must be  trivial when $|m|>1$. This contradicts the second assertion of Proposition \ref{not}.

\subsection{Case 2: every element of $\mathcal{Z}$ is nondegenerate} Choose $j_0$ as in Step 2 of Case 1. That is 
\begin{equation}
\label{ }
\Fix(\varphi^{p_j^2}) = \Fix{\varphi}
\end{equation}
for all $j \geq j_0$.  Fix a $j\geq j_0$ such that the points of  $\mathcal{Z}$ are nondegenerate fixed points of $\varphi^{p_j^2}$.
Note that there infinitely many such values of $j$.

To proceed we do not need to perturb $\varphi^{p_j^2}$ as in the previous case. We instead blow-up the points  $X$ and $Y$ in $\mathcal{I}$, whose existence is guaranteed by Lemma \ref{two},  to obtain a completion of $\varphi^{p_j^2}$ as a map of the closed annulus.  Gluing this map to itself again, we get an area preserving diffeomorphism $\Psi_j$ of the torus which is isotopic to the identity. The fixed point set of $\Psi_j$  now has the form
\begin{equation}
\label{ }
\left(\mathcal{Z}^+ \cup (\Ii \smallsetminus \{X,Y\})^+\right) \cup \left(\mathcal{Z}^- \cup (\Ii \smallsetminus \{X,Y\})^-\right).
\end{equation}
As described below, the following result again contradicts  Proposition \ref{hom}.

\begin{Proposition}\label{endpath2}
 If $j\geq j_0$ is sufficiently large then no contractible fixed point of $\Psi_j$ has Conley-Zehnder index equal to one. Moreover, any $W^+ \in \mathcal{Z}^+$ represents a nontrival class in $\FH(\Psi_j)$.
\end{Proposition}

\begin{proof}

Let $P^{\pm}$ be a contractible fixed point of $\Psi_j$ where $P$ denotes the corresponding fixed point of $\varphi$ and hence $\varphi^{p_j^2}$. 
Since $\pi_2(\T^2)$ is trivial, all classes of symplectic trivializations determined
by a spanning discs for $\Psi_{j,t}(P^{\pm})$  yield the same values of the mean index and Conley-Zehnder index of $P^{\pm}$.
So, in what follows we denote these simply as $\Delta(P^{\pm}; \Psi_{j,t})$ and $\mu(P^{\pm}; \Psi_{j,t})$. Since $P^{\pm}$ is contractible, $P$ must admit a spanning disc with image in $S^2 \smallsetminus \{X,Y\}$. Let $\Delta(P; \phi^t_H)$ denote the mean index computed with respect to the corresponding class of trivializations along $\phi^t_H(P)$. By \eqref{iter} we have 
\begin{equation}
\label{ktimes}
\Delta(P^{\pm}; \Psi_{j,t}) = p_j^2 \Delta(P; \phi^t_H).
\end{equation}

\noindent{\it Subcase 1: $P^{\pm} \in (\Ii \smallsetminus \{X,Y\})^{\pm}$.} In this case  $\Delta(P; \phi^t_H)$ is irrational,  and it follows from \eqref{ktimes} that  for large enough $j$ we have 
\begin{equation}
\label{ }
|\Delta(P^{\pm}; \Psi_{j,t})| = p_j^2 |\Delta(P;\phi^t_H)| > 2.
\end{equation} 
By \eqref{closer} it then follows that for sufficiently large $j$ we have 
\begin{equation*}
\label{up}
|\mu(P^{\pm}; \Psi_{j,t})| >1.
\end{equation*}

\noindent{\it Subcase 2: $P^{\pm} \in \mathcal{Z}^{\pm}$.} In this case $\Delta(P; \phi^t_H) = 0 \mod 4$. If $\Delta(P; \phi^t_H) \neq 0$
then we can argue as in the previous case to show that for sufficiently large $j$, we have $$ |\mu(P^{\pm}_{j}; \Psi_{j,t})| >1.$$ Otherwise, it follows from \eqref{ktimes} that $$\Delta(P^{\pm}; \Psi_{j,t}) = 0.$$ 
Since, by assumption, $P^{\pm}$ is nondegenerate, the strong form of \eqref{closer} applies and implies that 
\begin{equation*}
\label{down}
\mu(P^{\pm}; \Psi_{j,t}) = 0.
\end{equation*}
This settles the  first assertion of Proposition \ref{endpath2}. 

Fix a $W^+ \in \mathcal{Z}^+$ and let $W$  denote the corresponding fixed point of $\varphi$. Fix also a class $[{\xi}]$ of symplectic trivializations of $TS^2$ along $\phi^t_{H}(W)$. This determines an equivalence class $[\xi^{p_j^2}_{+}]$ of symplectic trivializations of $T\T^2$ along $\Psi_{j,t}(W^+)$.
This time let $P^{\pm}$ be any fixed point of $\Psi_j$ in the same homotopy class as $W^+$. Arguing as in Case 1, to prove the second assertion of Proposition \ref{endpath2} it suffices to show that for $j$ sufficiently large 
we have either 
\begin{equation*}
\label{ }
\Delta(P^{\pm}; \Psi_{j,t}, [\xi^{p_j^2}_+]) = \Delta(W^+; \Psi_{j,t}, [\xi^{p_j^2}_+])
\end{equation*}
or 
\begin{equation*}
\label{ }
|\Delta(P^{\pm}; \Psi_{j,t}, [\xi^{p_j^2}_+]) - \Delta(W^+; \Psi_{j,t}, [\xi^{p_j^2}_+])|>3.
\end{equation*}

\noindent{\it Subcase 3: $P^{\pm} \in (\Ii \smallsetminus \{X,Y\})^{\pm}$.} By our construction of $\Psi_j$ and  \eqref{iter} we have 
$$
\Delta(W^+; \Psi_{j,t}, [\xi^{p_j^2}_+])  =p_j^2\Delta(W; \phi^t_H, [\xi])
$$
and 
$$
\Delta(P^{\pm}; \psi_{j,t}, [\xi^{p_j^2}_+])  =p_j^2\Delta(P; \phi^t_H, [\xi]).
$$
Now $\Delta(W; \phi^t_H, [\xi]) = 0 \mod 4$ and  $\Delta(P; \phi^t_H, [\xi])$ is irrational, so for $j$ sufficiently large we have 
$$
| \Delta(W^+; \Psi_{j,t}, [\xi^{p_j^2}_+]) - \Delta(P^{\pm}; \Psi_{j,t}, [\xi^{p_j^2}_+])| > 3.
$$

\noindent{\it Subcase 4:  $P^{\pm} \in \mathcal{Z}^{\pm}$.} In this case, 
\begin{equation}
\label{ }
\Delta(W^+; \Psi_{j,t}, [\xi^{p_j^2}_+]) - \Delta(P^{\pm}; \Psi_{j,t}, [\xi^{p_j^2}_+]) =0 \mod 4.
\end{equation}
If the mean indices are not equal we can argue as in the previous subcase. If they are equal, then we are done. 
\end{proof}

Proposition \ref{endpath2} leads to the desired contradiction in Case 2. The first assertion  implies that $\Psi_j$ can not be Hamiltonian and the second implies that the Floer homology $\FH(\Psi_j)$ is nontrivial. This again contradicts 
 Proposition \ref{hom}.

\end{document}